\documentclass[12pt]{article}
\usepackage{a4wide}
\usepackage{amsfonts}
\usepackage[margin=2cm]{geometry}
\usepackage{graphics, amsmath,enumitem, comment, color}
\usepackage{graphicx}
\usepackage{epsfig}
\usepackage{amssymb}
\usepackage{amsmath,amsthm}
\usepackage{mathrsfs}
\usepackage{colortbl}
\usepackage{tikz}
\usetikzlibrary{positioning,arrows}
\usetikzlibrary{3d}
\usetikzlibrary[patterns]
\usepackage{pgflibraryarrows}
\usepackage{pgflibrarysnakes}
\usetikzlibrary{snakes}
\usepackage{arydshln}
\usepackage{graphicx}
\usepackage{setspace}
\usepackage{eucal}

\newif\ifcolorcomments
\newcommand{\allowcomments}[4]{
\newcommand{#1}[1]{\ifdraft{\ifcolorcomments{\textcolor{#4}{##1 --#3}}\else{\textsl{ ##1 \ --#3}}\fi}\else{}\fi}
}

\allowcomments{\comAB}{Ayreena}{A}{magenta}
\allowcomments{\comCC}{Carlo}{C}{orange}

\colorcommentstrue
\def\bc{\begin{center}}
\def\ec{\end{center}}
\def\be{\begin{equation}}
\def\ee{\end{equation}}

\def\N{\mathbb N}

\def\Q{\mathbb Q}

\def\R{\mathbb R}

\newcommand{\matr}[4]{
\left( \begin{array}{cc} #1 & #2 \\ #3 & #4 \end{array} \right)}

\newcommand{\0}{\mathbf 0}

\newcommand{\tb}{\Tilde{\beta}}

\newtheorem{lem}{Lemma}[section]

\newtheorem{dfn}[lem]{Definition}
\newtheorem{pro}[lem]{Proposition}
\newtheorem{thm}[lem]{Theorem}

\newtheorem{rem}[lem]{Remark}
\numberwithin{equation}{section}

\DeclareMathOperator\sign{sgn}
\newif\ifdraft\drafttrue

\allowdisplaybreaks
\usepackage{hyperref}

\title{Regularity properties of the $\alpha$-Wilton functions}

\author{Ayreena Bakhtawar
\thanks{A.B. acknowledges the support by Centro di Ricerca Matematica Ennio de
Giorgi, Scuola Normale Superiore di Pisa and UniCredit Bank R\&D division for financial support  under the project `Dynamics and Information Research Institute-Quantum Information (Teoria dell’Informazione), Quantum Technologies'. 
 A.B. would like to thank Centro di Giorgi for the excellent working conditions and the research travel support.
}
\and
Carlo Carminati
\thanks{
 C.C. is partially supported by the PRIN Grant 2022NTKXC of the Ministry of University and Research (MUR), Italy. C.C. acknowledges the support of the MIUR Excellence Department Project awarded to the Department of Mathematics, University of Pisa, CUP I57G22000700001.}
\and
Seul Bee Lee
\thanks{
S.L. is supported by the Institute for Basic Science (IBS-R003-D1).} 
}

\newcommand{\Addresses}{{
  \bigskip
  \footnotesize

  A.~Bakhtawar, \textsc{Centro di Ricerca Ennio De Giorgi, Scuola Normale Superiore, Piazza dei Cavalieri 3,
56126 Pisa, Italy}\par\nopagebreak
  \textit{E-mail address}, A.~Bakhtawar: \texttt{ayreena.bakhtawar@sns.it}

  \medskip

  C.~Carminati, \textsc{Dipartimento di Matematica, University di Pisa, Largo Bruno Pontecorvo 5, 56127
Pisa, Italy}\par\nopagebreak
  \textit{E-mail address}, C.~Carminati: \texttt{carlo.carminati@unipi.it}

  \medskip

  S.B.~Lee, \textsc{Center for Geometry and Physics, Institute for Basic Science (IBS), Pohang 37673, Korea }\par\nopagebreak
  \textit{E-mail address}, S.B.~Lee: \texttt{seulbee.lee@ibs.re.kr}

}
}
\begin{document}
\date{}
\maketitle

\begin{abstract}
The aim of this article is to study the regularity properties of the Wilton functions $W_\alpha$ associated with $\alpha$-continued fractions. We prove that the Wilton function is BMO for $\alpha\in[1-g,g]$ (where $g:=\frac{\sqrt{5}-1}{2}$ denotes the golden number), and we show that this result is optimal, since we find that on any left neighbourhood of $1-g$ and on any right neighbourhood of $g$ there are values $\alpha$ for which $W_\alpha$ is not BMO; the proof of this latter negative results exploits a special feature of the family of $\alpha$-continued fractions called ``matching''.
Our results complete those of Marmi--Moussa--Yoccoz (1997)  and of  Lee--Marmi--Petrykiewicz--Schindler (2024), where it is proven that  Wilton function is BMO for, respectively,  $\alpha=1/2$ (\cite{MaMoYo_97}) and $\alpha \in[\frac{1}{2},g]$ (\cite{LeMar_24}).
 
\end{abstract}

\section{Introduction}

%%%%%%%%%%%%%%%%%%%%%%%%%%%%%%%%%%%%%%%%

For $0\leq \alpha \leq 1$, let 
 $\bar{\alpha}=\max(\alpha,1-\alpha)$; the \emph{$\alpha$-continued fraction expansion} of a real number $x\in (0,\bar{\alpha})$ is associated to the iteration of the map $A_{\alpha}:(0,\bar{\alpha})\to[0,\bar{\alpha}]$ defined as follows:
 \begin{equation}\label{Amap}
 A_{\alpha}(x)=\left\vert\frac{1}{x}-\left[  \frac{1}{x}+1-\alpha\right]\right\vert,\; 
 \end{equation}
 where $[\cdot]$ denotes the integer part.
 
The family of maps $\{A_{\alpha}\}_\alpha$ was introduced by Nakada in \cite{Na81}, and as special cases it includes the standard continued fraction map when $\alpha=1,$ the nearest-integer continued fraction map when $\alpha=\frac{1}{2}$ and the by-excess continued fractions map when $\alpha=0$. For all $\alpha\in(0,1]$, these maps are expanding and admit a unique absolutely continuous invariant probability measure $d_{\mu_{\alpha}}=\rho_{\alpha}(x)dx$ whose density is bounded from above and below by a constant dependent on $\alpha.$ In the case $\alpha=0$, there is an indifferent fixed point and $A_{\alpha}$ does not have a finite invariant density but it preserves the infinite measure $d_{\mu_{0}}(x)=\frac{dx}{1-x}.$

The \emph{Wilton function} associated with an $\alpha$-continued fraction is defined as follows on $\mathbb{R} \setminus \mathbb{Q}$ 
\begin{equation}\label{aWFun}
W_{\alpha}(x)=\sum^{\infty}_{j=0} (-1)^j \beta_{\alpha,j-1}(x)\log x^{-1}_{\alpha,j}=\sum^{\infty}_{j=0} (-1)^j\beta_{\alpha,j-1}(x)\log(1/A_{\alpha}^{j}(x_{\alpha,0})),
\end{equation}
where the sequence  $x_{\alpha,n}=A^n_{\alpha}(x_{\alpha,0})$
with $x_{\alpha,0}=|x-\lfloor x+1-\bar{\alpha}\rfloor|$ and
$\beta_{\alpha,n}=x_{\alpha,0}x_{\alpha,1}\cdots x_{\alpha,n}$ for $n\geq 0,$
$\beta_{\alpha,-1}=1$.
When we consider $\alpha=1$, then $A_{1}$ is simply the Gauss map; in this case we shall often omit the dependence on $\alpha$, and write $x_n, \beta_n$ rather than $x_{\alpha,n}, \beta_{\alpha,n}$.

Note that the formula \eqref{aWFun} defines an $L^1(0,1)$ function which satisfies the functional equation
\begin{equation*}\label{BFE}
\begin{split}
& W_{\alpha}(x)=-\log(x)-xW_{\alpha}(A_{\alpha}(x))  ~~\text{ for  all } x\in(0,\bar{\alpha})\setminus\mathbb Q,\\
& W_{\alpha}(x)=W_{\alpha}(1-x) \qquad\qquad\qquad~\text{ for all }x\in(0,\min\{\alpha,1-\alpha\})\setminus\mathbb Q,
\end{split}
\end{equation*}
and more generally,
\begin{equation} \label{gen}
W_{\alpha}(x)=W_{\alpha}^{(K)}(x)+(-1)^{K+1}\beta_{\alpha,K}(x)W_{\alpha}(A_{\alpha}^{K+1}(x)) \quad (K\in \N, x\in(0,\bar{\alpha})\setminus \Q),
\end{equation}
where $W^{(K)}$ denotes the partial sum 
\begin{equation}\label{PS}
W_{\alpha}^{(K)}(x)=\sum_{j=0}^{K}(-1)^j \beta_{\alpha,j-1}{(x)}\log(1/A_{\alpha}^{j}(x))
\end{equation}
with respect to the $\alpha$-continued fraction.
 
The series $\eqref{aWFun}$ was first introduced by Wilton \cite{Wi_33} for $\alpha=1$ in order the study of trigonometric series 
\begin{equation} \label{TS}
    \phi_{1}(x)=-\frac{1}{\pi}\sum_{n\geq1}\frac{\tau(n)}{n}\sin(2\pi n x),
\end{equation}
where $\tau(n)$ is the number of divisors of the natural number $n.$
Indeed, the author showed that the series \eqref{TS} converges if and only if $W_1$ is convergent.
The series \eqref{aWFun} defining $W_1$ as
\begin{equation}\label{WFun}
W_1(x)=\sum^{\infty}_{j=0} (-1)^j\beta_{j-1}(x)\log x^{-1}_{j}.
%=\sum^{\infty}_{j=0} \beta_{j-1}(x)\log(1/A_{\alpha}^{j}(x)),
\end{equation}

We will refer to the irrational real numbers $x$ for which the series \eqref{aWFun} converges as
the \emph{Wilton numbers}. It can be proved that the  series \eqref{WFun} converges if and only if it fulfills the \emph{Wilton condition}
\begin{equation*}
\left\vert  \sum_{j=0}^{\infty}  (-1)^{j}  \frac{\log(q_{j+1}(x))} {q_{j}(x)}    \right\vert<\infty,
\end{equation*}
  where $q_{j}$ denotes the  denominator of the $j$th convergent of $x$ associated with the Gauss map $A_1$.

All Diophantine numbers, i.e. $ x\in\R\setminus\Q$ 
 such that $q_{n+1}=O(q_{n}^{1+\tau})$ where $\tau\geq 0,$ are Wilton numbers. 
Note that the Wilton function \eqref{WFun} is an alternating sign version of the Brjuno function, introduced by Yoccoz in 1988, which plays an important role in the theory of dynamical systems, more precisely
in the study of iteration of a quadratic polynomials (for more details on the Brjuno function, see \cite{Mar_90, MaMoYo_01, MaMoYo_06}). 

Clearly, all Brjuno numbers are Wilton numbers but not vice versa. Whereas the Hausdorff dimension of the difference set is $0$, i.e. $\mathrm{dim}_H(\mathcal W \setminus \mathcal B)=0$, where $\mathcal W$ and $\mathcal B $ denote the set of Wilton and Brjuno numbers respectively. This follows from the fact that  $(\mathcal W\setminus \mathcal B)\subset (\R \setminus \mathcal B)=\mathcal B^{c},$ and the Hausdorff dimension of the set $\mathcal B^{c}$ is $0$ as it is properly contained in the union of the set of Liouville numbers and the set of rational numbers. 

 In recent years, Balazard--Martin 
\cite{BaMa_19} studied the Wilton function $W_1$ in terms of its convergence properties and in the context of the Nyman and Beurling criterion \cite{BaMa_12,BaMa_13}. For example, in \cite{BaMa_13}, the authors reduced the study of the autocorrelation function to that of the Wilton function $W_{1}$ in order to show that the points of differentiability of the autocorrelation function $\mathcal{A}(\lambda)=\int_{0}^{\infty}\{ t\}\{\lambda t \}\frac{dt}{t^2}$ are the positive irrational numbers such that the series $\sum_{j\geq 0}(-1)^{j+1}\frac{\log q_{j+1}}{q_{j}}$ converges.

% Instead of considering Wilton function with respect to Gauss map as in \eqref{WFun}, it can also be defined for $\alpha$-continued fraction.

%

The aim of this paper is to study BMO regularity properties of Wilton functions associated with $\alpha$-continued fractions for $\alpha\in(0,1).$
In \cite{MaMoYo_97}, Marmi--Moussa--Yoccoz proved that $W_{1/2}$ is in BMO. 
Recently, the third author together with Marmi, Petrykiewicz and Schindler \cite{LeMar_24} improved this result by studying the regularity properties of Wilton function. In particular, they showed in \cite{LeMar_24} that $W_\alpha\in \mathrm{BMO}$ for all $\alpha \in [1/2,g]$, where $g=\frac{\sqrt 5 -1}{2}$. The aim of this article is to further 
improve this result of Lee--Marmi--Petrykiewicz--Schindler by extending the interval of $\alpha$. 

\begin{figure}
    \centering
    \includegraphics[width=0.8\linewidth]{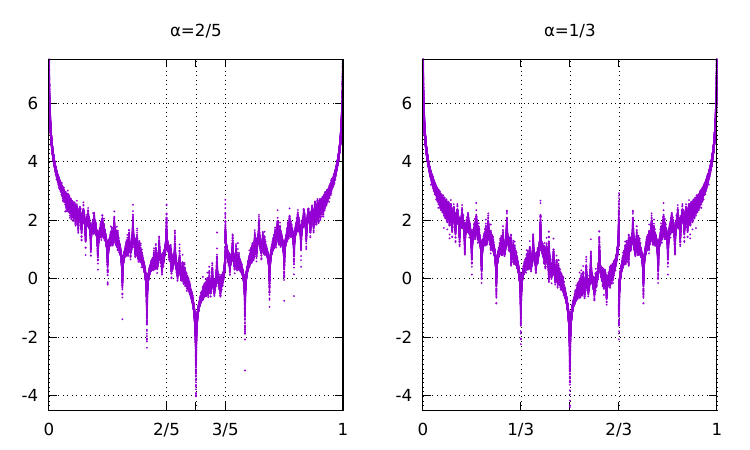}
    \caption{The graph of $W_{2/5}$ and $W_{1/3}$: the first is in BMO, the latter isn't (as one might guess observing the ``blow up'' at $x=2/3$).}
    \label{fig:W25vsW13}
\end{figure}

Our first main result is as follows:
\begin{thm}\label{main}
The Wilton function $W_\alpha\in \mathrm{BMO}$ for all $\alpha \in [1-g,g]$.
\end{thm}
We will also show that this result is optimal:
\begin{thm}\label{optimal}
\begin{enumerate}
\item[\rm(i)] If $\alpha \in (g,1]\cap \mathbb{Q}$, then $W_\alpha$ is not in $\mathrm{BMO}$.
\item[\rm(ii)] There exists a sequence $(u_m)_m$ of rational values, $u_m \uparrow 1-g$ as $m \to +\infty$ such that $W_{u_m}$ is not in $\mathrm{BMO}$.
\end{enumerate}
\end{thm}

It is interesting to point out that all these results (both in the positive and negative direction) are strictly linked with a remarkable feature of $\alpha$-continued fractions called {\em matching}; the relevance of this property was first pointed out by \cite{NaNa_08} in relation to the study of the entropy of $\alpha$-continued fractions, and it lead to several results in this field (see \cite{CT12}).
In fact, the technique we use to prove our main result can be adapted to prove, in a simple way, that the entropy of $\alpha$-continued fractions is constant on the interval $[1-g,g]$ (see appendix \ref{apx}).

\section{Notations and preliminary results }\label{notation}

\subsection{Folded $\alpha$-continued fractions}
Fix $\alpha\in(0,1],$ let $\bar{\alpha}=\max(\alpha,1-\alpha)$ and consider the map $A_{\alpha}:[0,\bar{\alpha}]\to[0,\bar{\alpha}]$ be the transformation of $\alpha$-continued fraction
 defined by $A_\alpha(0)=0$ and
 \begin{equation}\label{Amap}
 A_{\alpha}(x)=\left\vert\frac{1}{x}-\left[  \frac{1}{x}\right]_{\alpha}\right\vert, \; \; 
 \end{equation}
 for $x\in(0,\bar\alpha]$, where 
 $[x]_{\alpha}=[x+1-\alpha]$ and $\left[ \cdot \right]$ denotes the integer part.
  Put
$x_{\alpha,0}=|x-[x]_{\bar{\alpha}}|,$ $a_{\alpha,0}=[x]_{\bar{\alpha}},$     $\epsilon_{\alpha,0}(x)=\sign{(x-[x]_{\bar{\alpha}})}$ and define by recurrence for $n\geq 0:$
\begin{align*}
x_{\alpha,n+1}&=%A^{n+1}_{\alpha}(x_{\alpha,0})=
A_{\alpha}(x_{\alpha,n}),~
a_{\alpha,n+1}(x)=\left[ \frac{1}{x_{\alpha,n}}\right]_{\alpha} \text{ and }
\epsilon_{\alpha,n+1}=\sign\left(  \frac{1}{x_{\alpha,n}}-\left[ \frac{1}{x_{\alpha,n}}\right]_{\alpha}  \right).
\end{align*}
The $\alpha$-continued fraction expansion of $x$ is
$$x = a_{\alpha,0}+\dfrac{\epsilon_{\alpha,0}}{a_{\alpha,1}+\dfrac{\epsilon_{\alpha,1}}{\ddots+\dfrac{\epsilon_{\alpha,n-1}}{a_{\alpha,n}+\dfrac{\epsilon_{\alpha,n}}{\ddots}}}}.$$
Let $\frac{p_{\alpha,n}}{q_{\alpha,n}}$ be the $n$th finite truncation of this expansion, that is,
\begin{equation}\label{eq:pc}\frac{p_{\alpha,n}}{q_{\alpha,n}} = a_{\alpha,0}+\dfrac{\epsilon_{\alpha,0}}{a_{\alpha,1}+\dfrac{\epsilon_{\alpha,1}}{\ddots+\dfrac{\epsilon_{\alpha,n-1}}{a_{\alpha,n}}}}.
\end{equation}
It is called the $n$th \emph{convergent} of $x$.
Let $p_{\alpha,-1}=1$, $q_{\alpha,-1}=0$ for the convenience. 
%Note that $p_{0}= a_0 \sqrt 2$, $q_{0}=1$.

Thanks to the isomorphism between   $2\times 2$ matrices and fractional transformations the following notation will be useful
\begin{equation*}
\begin{pmatrix} a & b \\ c & d\end{pmatrix} \cdot x =\frac{ax+b}{cx+d}.
\end{equation*}
Then equation \eqref{eq:pc} induces, for $n\geq 1$, 
\begin{equation}\label{matrix}
\begin{pmatrix} p_{\alpha,n-1} & p_{\alpha,n}\\ q_{\alpha,n-1} & q_{\alpha,n}\end{pmatrix} 
= \begin{pmatrix}1&a_{\alpha,0}\\0&1\end{pmatrix}
\begin{pmatrix} 0  & \epsilon_{\alpha,0} \\ 1 & a_{\alpha,1} \end{pmatrix} 
\begin{pmatrix} 0  & \epsilon_{\alpha,1} \\ 1 & a_{\alpha,2} \end{pmatrix}
\cdots\begin{pmatrix} 0  & \epsilon_{\alpha,n-2} \\ 1 & a_{\alpha,n-1} \end{pmatrix}\begin{pmatrix} 0  & \epsilon_{\alpha,n-1} \\ 1 & a_{\alpha,n} \end{pmatrix}.
\end{equation}
By applying the above matrices on the point $\epsilon_{\alpha,n}x_{\alpha,n}$, we have $$\frac{p_{\alpha,n}+\epsilon_{\alpha,n}p_{\alpha,n-1}x_{\alpha,n}}{q_{\alpha,n}+\epsilon_{\alpha,n}q_{\alpha,n-1}x_{\alpha,n}}=x.$$
By calculating the determinant of the matrices of \eqref{matrix}, it is immediate that for $n\geq1$
\begin{equation}\label{eq:det}
p_{\alpha,n} q_{\alpha,n-1} - q_{\alpha,n} p_{\alpha,n-1}
= (-1)^{n} \epsilon_{\alpha,0}\epsilon_{\alpha,1} \epsilon_{\alpha,2} \cdots \epsilon_{\alpha,n-1}.
\end{equation}
Thus the convergents of $x$ satisfy the following recursive relation:
\begin{equation}\label{qrecurrence}
 p_{\alpha,n}=a_{\alpha,n}
p_{\alpha,n-1}+\epsilon_{\alpha,n-1}p_{\alpha,n-2}, \ \ \ 
q_{\alpha,n}=a_{\alpha,n} 
q_{\alpha,n-1}+\epsilon_{\alpha,n-1}q_{\alpha,n-2}.
\end{equation}
%and that it satisfies the well known recurrence equations for %$n\geq0$
%\begin{align}\label{recurrence}
%p_{n+1} = a_{n+1}  p_{n} + \epsilon_{n+1} p_{n-1}, \qquad
%q_{n+1} = a_{n+1} q_{n} + \epsilon_{n+1} q_{n-1}.
%\end{align}
%
It follows 
\begin{equation}\label{eq:x-p/q}
x-\frac{p_{\alpha,n}}{q_{\alpha,n}}=\frac{p_{\alpha,n}+\epsilon_{\alpha,n}p_{\alpha,n-1}x_{\alpha,n}}{q_{\alpha,n}+\epsilon_{\alpha,n}q_{\alpha,n-1}x_{\alpha,n}} -\frac{p_{\alpha,n}}{q_{\alpha,n}}
=\frac{(-1)^{n+1} \epsilon_{\alpha,0}\epsilon_{\alpha,1} \epsilon_{\alpha,2} \cdots \epsilon_{\alpha,n-1}\epsilon_{\alpha,n}x_{\alpha,n}}{q_{\alpha,n}(q_{\alpha,n}+\epsilon_{\alpha,n}q_{\alpha,n-1}x_{\alpha,n})}
\end{equation}
and 
\begin{equation*}
\sign\left(x-\frac{p_{\alpha,n}}{q_{\alpha,n}}\right)=\sign(q_{\alpha,n}x-p_{\alpha,n})=(-1)^{n+1}\epsilon_{\alpha,0}\epsilon_{\alpha,1} \epsilon_{\alpha,2} \cdots \epsilon_{\alpha,n}.
\end{equation*}

Define
$\beta_{\alpha,n}:=\prod^{n}_{i=0}x_{\alpha,i}=\prod^{n}_{i=0}A^{i}_{\alpha}(x_{\alpha,0})$ for $n\geq0$  as the product of the iterates along the $A_{\alpha}$-orbit with $\beta_{\alpha,-1}=1.$ 
From \cite[Lemma 1]{LuMaNaNa_10}, for all $n\geq 1$ we have $\beta_{\alpha,n}=|q_{\alpha,n}x-p_{\alpha,n}|$.
By definition, $x_{\alpha,n}=\frac{\beta_{\alpha,n}}{\beta_{\alpha,n-1}}$.
Combining with \eqref{eq:x-p/q}, we have
\begin{align*}
\beta_{\alpha,n}=\frac{\beta_{\alpha,n+1}}{x_{\alpha,n+1}}=\frac{1}{q_{\alpha,n+1}+q_{\alpha,n}\epsilon_{\alpha,n+1}x_{\alpha,n+1}}.
\end{align*}
Since $q_{\alpha,n+1}>q_{\alpha,n}>0$ (\cite[Lemma 1]{MoCaMa_99}) and $\epsilon_{\alpha,n+1}x_{\alpha,n+1}=\frac{1}{x_{\alpha,n}}-[\frac{1}{x_{\alpha,n}}]_\alpha\in[\alpha-1,\alpha)$,
for $\alpha>0$, we have 
\begin{equation*}
\frac{1}{1+\alpha}<\beta_{\alpha,n}q_{\alpha,n+1}<\frac{1}{\alpha}.
 \end{equation*}
\begin{pro}[{\cite[Lemma 3]{MoCaMa_99}}]\label{prop}
 Let $\alpha>0$ and $\bar{\alpha}=\max(\alpha,1-\alpha).$ Then for all $n\geq1$ one has 
% \begin{itemize}
% \item[\rm{(i)}] $q_{\alpha,n+1}>q_{\alpha,n}>0;$
% \item[\rm{(ii)}] $p_{\alpha,n}>0$ when $x>0$ and $p_{\alpha,n}<0$ when $x<0;$
 % \item [\rm{(iii)}]$|q_{n}x-p_{n}|=\frac{1}{q_{n+1}+\epsilon_{n+1}q_{n}{x_{n+1}}}$ so that $\frac{1}{1+\alpha}<\beta_{n}q_{n+1}<\frac{1}{\alpha};$

% \item [\rm{(ii)}] 
$$ \begin{array}{ccc}
    \beta_{\alpha,n}&\leq& \bar{\alpha} \rho^{n}_{\alpha},    \\
      1/q_{\alpha,n+1} &<& (1+\alpha)\bar{\alpha}\rho^{n}_{\alpha}, 
 \end{array}\text{
where }%$\rho_{\alpha}<1$: 
\rho_{\alpha}=\begin{cases}
    g & g<\alpha\leq 1,\\
    \sqrt{2}-1 & \sqrt{2}-1 \leq \alpha \leq g,\\
    \sqrt{1-2\alpha}&0<\alpha<\sqrt{2}-1. 
\end{cases}$$

 %$\rho_{\alpha}=g$ for $g<\alpha\leq 1;$
 %$\rho_{\alpha}={\gamma}$ for $\gamma \leq \alpha \leq g;$
 %$\rho_{\alpha}=\sqrt{1-2\alpha}$ for $0<\alpha<\gamma.$

%  \end{itemize}
 \end{pro}

%\subsection{Folded vs unfolded $\alpha$-continued fractions}

\subsection{Unfolded $\alpha$-continued fractions and matching}

In this subsection, we recall another variant of  $\alpha$-continued fractions (called {\it unfolded} $\alpha$-continued fractions), and we shall show that the two algorithms have the same features. In particular, the folded and unfolded algorithms lead essentially to the same Wilton function, and we shall use the unfolded version of the algorithm in order to directly use the results about matching (results which have been developed in the unfolded setting).

 Following \cite{NaNa_08}, consider the family of maps $(T_\alpha)_{\alpha \in [0,1]}$, $T_\alpha: [\alpha-1,\alpha) \to[\alpha-1,\alpha)$  defined by $T_\alpha(0)=0$ and $$T_{\alpha}(x) = \frac{\epsilon(x)}{x} - c_{\alpha}(x) \qquad \textup{for }x \neq 0$$ 
with 
$$\epsilon(x) := \textup{sgn}(x) \qquad c_{\alpha}(x) := 
\left[ \frac{1}{|x|} + 1 - \alpha \right].$$
We also set
\begin{equation*}
\tilde{\epsilon}_{\alpha,n}= \tilde{\epsilon}_{\alpha,n}(x)=\epsilon( T^{n-1}_{\alpha}(x)), \ \ x_{\alpha,n}=T^{n}_{\alpha}(x) \ \ \mbox{ and }
c_{\alpha,n}= c_{\alpha,n}(x)=c_\alpha ( T^{n-1}_{\alpha}(x) ).
\end{equation*}
With these notations, we have
\begin{equation}\label{finite}
    x = \dfrac{\tilde{\epsilon}_{\alpha,1}}
{c_{\alpha,1}+ \dfrac{\tilde{\epsilon}_{\alpha,2}}
{\ddots+       \dfrac{\tilde{\epsilon}_{\alpha,n}}
{c_{\alpha,n}+x_{\alpha,n}}}}= \dfrac{\tilde{\epsilon}_{\alpha,1}}
{c_{\alpha,1}+ \dfrac{\tilde{\epsilon}_{\alpha,2}}
{\ddots+       \dfrac{\tilde{\epsilon}_{\alpha,n}}
{c_{\alpha,n}+\ddots}}}.
\end{equation}
The rightmost expression above is called  the \emph{infinite (unfolded) $\alpha$-continued fraction expansion of $x$}.

As in the folded version, by setting
\begin{equation}\label{matrix2}
M_{\alpha,x,n}:= \begin{pmatrix} 0  & \tilde{\epsilon}_{\alpha,1} \\ 1 & c_{\alpha,1} \end{pmatrix} 
\begin{pmatrix} 0  & \tilde{\epsilon}_{\alpha,2} \\ 1 & c_{\alpha,2} \end{pmatrix}
\cdots\begin{pmatrix} 0  & \tilde{\epsilon}_{\alpha,n} \\ 1 & c_{\alpha,n} \end{pmatrix},
\end{equation}
we can rewrite equation \eqref{finite} as $x=M_{\alpha,x,n}\cdot x_{\alpha,n}$ or, writing the entries of $M_{\alpha,x,n}$ explicitly,
\begin{equation*}
x=\frac{\tilde{p}_{\alpha,n-1}x_{n}+\tilde{p}_{\alpha,n}}{\tilde{q}_{\alpha,n-1}x_{n}
+\tilde{q}_{\alpha,n}}, \ \ \ \ \mbox{ where } M_{\alpha,x,n}=\begin{pmatrix} \tilde{p}_{\alpha,n-1}(x) & \tilde{p}_{\alpha,n}(x)\\ \tilde{q}_{\alpha,n-1}(x) & \tilde{q}_{\alpha,n}(x)\end{pmatrix}
\end{equation*}
and $\frac{\tilde{p}_{\alpha,n}} {\tilde{q}_{\alpha,n}}:= M_{\alpha,x,n}\cdot 0 $, which corresponds to the truncated $\alpha$-continued fraction (or {\it convergent}) of order $n$. 
In a similar way to obtain \eqref{eq:x-p/q}, the following approximation identity holds
\begin{equation*}
\left\vert x-\frac{\tilde{p}_{\alpha,n}} {\tilde{q}_{\alpha,n}}           \right\vert  =\frac{|x_{n}|} {\tilde{q}_{\alpha,n}(\tilde{q}_{\alpha,n}
+\tilde{q}_{\alpha,n-1}x_{n})}.
\end{equation*}

\subsubsection{Folded vs. Unfolded algorithms}

The map $A_{\alpha}$ is just the folded version of $T_{\alpha}$:
 the families $(T_\alpha)_\alpha$ and $(A_\alpha)_\alpha$ are semiconjugated by the map $x\mapsto |x|$, namely
%\begin{eqnarray}\label{semiconjugacy}
 %   |T_\alpha(x)|=A_\alpha(|x|), \ \ \ x \in [\alpha-1,\alpha)
%\end{eqnarray}
\begin{eqnarray}\label{Ksemc}
    |T^K_\alpha(x)|=A^K_\alpha(|x|), \ \ \ x \in [\alpha-1,\alpha), \ \ K \in \mathbb{N},
\end{eqnarray}
and they are associated to a pair of continued fraction expansion called respectively {\it unfolded} and {\it folded} $\alpha$-continued fractions. 

For $x\in [\alpha-1,\alpha)$, we can define 
$\Tilde{\beta}_{\alpha,n}(x):=\prod^{n}_{i=0}|T^{i}_{\alpha}(x)|=\beta_{\alpha,n}(|x|)$, and also the  Wilton function associated to the unfolded algorithm, which is the one periodic function $\tilde{W}_\alpha$ which satisfies
\begin{equation}\label{Twilton}
\tilde{W}_{\alpha}(x)=\sum^{\infty}_{j=0} (-1)^j \tb_{\alpha,j-1}(x)\log(1/|T_{\alpha}^{j}(x)|), \ \ \ x\in [\alpha-1,\alpha).
\end{equation}
It is immediate to check that
\begin{equation}\label{twvsw}
    \tilde{W}_\alpha(x)=W_\alpha(|x|) \ \ \ \mbox{ for } x\in [\alpha-1,\alpha).
\end{equation}
This means that $\tilde{W}_\alpha(x)=W_\alpha(x)$ when $\alpha \geq 1/2$ (the two periodic function agree on $[0,\alpha)$ and by  symmetry also on $[\alpha-1,0]$); on the other hand for $\alpha<1/2$, one has that $\tilde{W}_\alpha(x)=W_\alpha(-x)$ (indeed, this identity holds on $[\alpha-1,0]$, and by symmetry also on $(0,\alpha)$).
Obviously, the regularity properties of $\tilde{W}_\alpha$ and $W_\alpha$ are the same, and since all the results about matching are stated for the family $T_\alpha$, in Section \ref{tre}, we will prefer to work in the unfolded setting.
As in the folded case, also the unfolded Wilton function satisfies a functional equation 
\begin{equation} \label{gen}
\tilde{W}_{\alpha}(x)=\tilde{W}_{\alpha}^{(K)}(x)+(-1)^{K+1}\Tilde{\beta}_{\alpha,K}(x)\tilde{W}_{\alpha}(T_{\alpha}^{K+1}(x)) \quad (K\in \N,~ x\in([\alpha-1,\alpha)\setminus \Q),
\end{equation}
where $\tilde W_\alpha^{(K)}$ is the partial sum $\sum_{j=0}^K (-1)^j \tilde\beta_{\alpha,K-1}(x)\log(1/|T_{\alpha}^j(x)|)$.)

\subsubsection{Matching property}

We now recall the {\it matching property} first discovered by \cite{NaNa_08} in connection with the study of the metric entropy of $T_\alpha$\footnote{Actually we shall  follow the notation introduced in \cite{CT12} (which is slightly different from  the original in \cite{NaNa_08}).}; in fact we will see that this matching property plays an important role also for the regularity properties of the Wilton function.

\begin{dfn}
   The value $\alpha \in (0, 1]$ is said to satisfy an {\it algebraic matching condition} of order $(n, m)$, denoted by $(n,m)_{\rm alg}$,
when the following matrix identity holds:
\begin{equation}\label{algmatch}
(n,m)_{\rm alg}: \qquad M_{\alpha, \alpha, n} = \matr{1}{1}{0}{1} M_{\alpha, \alpha-1, m} \matr{-1}{0}{1}{1} . 
\end{equation}  
\end{dfn}

To get some intuition of what this condition means from a dynamic point of view, 
one should note that $(n,m)_{\rm alg}$ implies
$$T_\alpha^{n+1}(\alpha) = T_\alpha^{m+1}(\alpha-1)$$
(see \cite[Appendix A1]{CT12}, and also the brief explanation on the next page). For this reason, the pair $(n,m)$ satisfying $(n,m)_{\rm alg}$ is referred to as {\it matching exponents}; the difference $m-n$ is called {\it matching index}.
Actually in \cite{CT12} it is proved that the set
$$\mathcal{M}_{\rm alg} := \{ \alpha \in (0, 1]:\exists \ n, m
    \in \mathbb{N} \  \textup{ s.t. }  \alpha \textup{ satisfies } (n,m)_{\rm alg} \}$$
contains an open neighbourhood of $(0,1]\cap \Q$ of full measure; the connected components of this open set are called {\it matching intervals}, and on any matching interval, both sides of \eqref{algmatch} are constant (see \cite[Lemma 3.7]{CT12}).
Any matching interval $J$ contains a unique rational value $p/q$ with a minimal denominator called the {\it pseudocenter} of $J$; moreover, the matching exponents $(n,m)$ can be easily extracted from the even length continued fraction expansion of its pseudocenter: indeed if $p/q$ is a pseudocenter, then by choosing the continued fraction expansion $[0;a_1,a_2,\cdots,a_\ell]$ with even length $\ell$ from its two possible expansions, the matching exponents $(n,m)$ of $J$ are 
$$n := \sum_{j: \  \textup{even}}a_j  \qquad\text{ and }\qquad m := \sum_{j: \  \textup{odd}} a_j,$$
i.e. every $x \in J$ satisfies  the matching condition $(n,m)_{\rm alg}$ (see \cite[Theorem 3.1]{CT12}).

In \cite{BCIT13}, a more explicit description of $\mathcal{M}_{\rm alg}$ is given in terms of the Gauss map $T_1$: indeed $[0,1]\setminus \mathcal{M}_{\rm alg}= \mathcal{E}$, where
\begin{equation}\label{exceptional}
 \mathcal{E}:=\{x : \ T_1^k(x)\geq x \ \  \forall k\in \N\}.
 \end{equation}
 Note that $\mathcal{E}$ is a zero measure set, but $\mathrm{dim}_H(\mathcal{E})=1$.

The Gauss map $T_1$ can also be used to characterize those rational values $p/q$ which are the pseudocenter of some matching interval $J$: indeed, this happens if and only if $T_1^k(p/q)\notin (0,p/q)$ for all $k\in \N$. 
Let us give a few examples of this phenomenon.
\begin{enumerate}
    \item The interval $(g, 1]$ is a matching  interval of index $-1.$
    \item The interval $(1-g,g)$ contains infinitely many matching intervals, all of index 0; the largest one is the rightmost one, namely $(\sqrt{2}-1, g)$; however, $\mathrm{dim}_H(\mathcal{E}\cap [1-g,g])>0$ (see \cite{CT13}).
    \item Every left neighbourhood of $1-g$ contains infinitely many matching intervals of index\footnote{
    In fact every left neighbourhood of $1-g$ contains infinitely many matching intervals of any index (see \cite{CT13}).
    } $+1$: indeed any rational value of the type 
    $u_m = [0;2,1^{2m-1}]$ (with a tail of $2m-1$ ones) is the pseudocenter of a matching interval on the left of $1-g$; these intervals accumulate on $1-g$ as $m\to +\infty$.  
\end{enumerate}

We conclude this small subsection with a remark that will play an important role in the following discussion.
It is known\footnote{See \cite[Appendix A1]{CT12}.} that the condition \eqref{algmatch} implies that
$$\frac{1}{T_\alpha^n(\alpha)}+\frac{1}{T_\alpha^m(\alpha -1)}=-1$$
This implies that the terms on the left side of the above sum have opposite signs, and 
if $
T_\alpha^{n+1}(\alpha)=\frac{\epsilon}{T_\alpha^n(\alpha)}-c$, then
$$\frac{1}{|T_\alpha^m(\alpha-1)|}=-\frac{\epsilon}{T_\alpha^m(\alpha-1)}=\epsilon + \frac{\epsilon}{T_\alpha^n(\alpha)}=\epsilon+c+T_\alpha^{n+1}(\alpha)
$$
and this last equality implies that
$$ T_\alpha^{m+1}(\alpha-1)=\frac{1}{|T_\alpha^m(\alpha-1)|} -\epsilon -c = T_\alpha^{n+1}(\alpha).$$
This last equality corresponds to the following matrix identity
$$ M_{\alpha, \alpha-1,m+1}^{-1} \begin{pmatrix}1&1\\0&1\end{pmatrix}^{-1}=M_{\alpha, \alpha,n+1}^{-1}.
$$

From the above identity, we also get that 
\begin{equation}\label{pnqn}
    \tilde{p}_{\alpha,n}(\alpha)=\tilde{p}_{\alpha,m}(\alpha-1) + \tilde{q}_{\alpha,m}(\alpha-1),
\ \  \ \ \ \
\tilde{q}_{\alpha,n}(\alpha)=\tilde{q}_{\alpha,m}(\alpha-1).
\end{equation}
Moreover, if $\alpha$ is rational, then one has that, for sufficiently small $\varepsilon>0$,
\begin{equation}\label{eq:MM}
   \begin{array}{cc}
  M_{\alpha, x-1,m+1}= M_{\alpha, \alpha-1,m+1}    & \mbox{ for } \alpha  < x < \alpha +\varepsilon, \\
   M_{\alpha, x,n+1}= M_{\alpha, \alpha,n+1}   &
   \mbox{ for } \alpha -\varepsilon < x < \alpha,
\end{array} 
\end{equation}
and this implies that, setting $\phi(x):=M_{\alpha, \alpha,n+1}^{-1} \cdot x$, for sufficiently small $\varepsilon >0 $ , we have 
\begin{equation}\label{phalpha}
    \begin{array}{cc}
  T_\alpha^{m+1}(x-1) = \phi(x)   & \mbox{ for } \alpha  < x < \alpha +\varepsilon, \\
   T_\alpha^{n+1}(x) = \phi(x),    &
   \mbox{ for } \alpha -\varepsilon < x < \alpha.
\end{array}
\end{equation}
From \eqref{pnqn} and \eqref{eq:MM}, we also get that there is an analytic function $b(x) $ such that
\begin{equation}\label{balpha}
    \begin{array}{cc}
  \tilde{\beta}_{\alpha,m}(x-1)=b(x)    & \mbox{ for } \alpha  < x < \alpha +\varepsilon, \\
   \tilde{\beta}_{\alpha,n}(x)=b(x)    &
   \mbox{ for } \alpha -\varepsilon < x < \alpha.
\end{array}
\end{equation}

\begin{lem}\label{matchlink}
Let $J$ be a matching interval with matching exponents $(m,n)$ and let $\alpha \in J \cap \Q$.
Then there exist a neighbourhood $U$ of $\alpha$,   functions $\beta, \phi$ which are smooth on $U$ and $h \in L^\infty(U)$ such that
\begin{equation}\label{localal}
 \tilde{W}_\alpha(x)=h(x)+{\mathrm{sgn}(\alpha-x)}^{n-m}\beta(x) \tilde{W}_\alpha(\phi(x)) \ \ \ \ \mbox{ for } x \in U.   
\end{equation}

\end{lem}

\begin{proof}
% {\bf Lemma \ref{localal}}
We will use the functional equation \eqref{gen}.
If $x$ is in a left neighbourhood of $\alpha$, then we can write
$$\tilde{W}_\alpha(x)=\tilde{W}^{(n)}_\alpha(x) +(-1)^{n+1}\tilde{\beta}_{\alpha,n}(x)\tilde{W}_\alpha(T_\alpha^{n+1}(x)).$$
On the other hand, if $x$ belongs to a right neighbourhood of $\alpha$, then we can write 
$$\tilde{W}_\alpha(x)=\tilde{W}_\alpha(x-1)=\tilde{W}_\alpha^{(m)}(x-1)+(-1)^{m+1}\tilde{\beta}_{\alpha,m}(x-1)\tilde{W}_\alpha(T_\alpha^{m+1}(x-1)).$$
Therefore, we can use \eqref{balpha} and \eqref{phalpha} to conclude that for $x\in U=(\alpha -\varepsilon,\alpha+\varepsilon)$, we have
$$\tilde{W}_\alpha(x)=h(x)+{\mathrm{sgn}(\alpha - x)}^{n-m}\beta(x) \tilde{W}_\alpha(\phi(x)),$$
where
$$
h(x)=\tilde{W}^{(n)}_\alpha(x) \chi_{(\alpha-\varepsilon, \alpha)}(x)+
\tilde{W}_\alpha^{(m)}(x-1)\chi_{(\alpha, \alpha+\varepsilon)}(x)\text{ and } \ \ \ 
\beta(x) =\pm b(x).$$
\end{proof}

\section{Behaviour of $\tilde{W}_{\alpha}$ near rational points}\label{tre}

\begin{lem}\label{Wzero}
Let $\alpha\in (0,1]$, then 
%$$ \int_{0}^{x}\tilde{W}_\alpha (t)dt = -x\log x + x+ O(x^2)  \ \ \ \mbox{ as } x\to 0^+$$
$$ \int_{0}^{x}\tilde{W}_\alpha (t)dt = -x\log x + x+ o(x)  \ \ \ \mbox{ as } x\to 0^+.$$
\end{lem}
\begin{proof}
The proof of this follows directly integrating the functional equation  $\tilde{W}_\alpha(x)=-\log x -x\tilde{W}_\alpha(T_\alpha(x))$: the term $-x\log x +x$ comes from the integration of $-\log x$, while the integration of $x\tilde{W}_\alpha(T_\alpha(x))$ leads to a term which is $o(x)$ as $x\to 0$ because the function $\tilde{W}_\alpha\circ T_\alpha$ is in $L^1$. Indeed, this last property follows directly from the fact that the invariant measure  has a BV density $d\mu_\alpha(x)=\rho_\alpha(x) dx$ such that $0<m\leq \rho_\alpha(x)\leq M$, and since 
$$\int_{\alpha-1}^\alpha  |\tilde{W}_\alpha| \circ T_\alpha (x) \rho_\alpha(x) dx=    \int_{\alpha-1}^\alpha  |\tilde{W}_\alpha(x)|  \rho_\alpha(x) dx,$$
we easily get that $\|\tilde{W}_\alpha \circ T_\alpha\|_1\leq \frac{M}{m}\|\tilde{W}_\alpha\|_1$, where $\|\cdot\|_1$ is the $L^1$-norm with respect to the Lebesgue measure.
The fact $\tilde W_{\alpha}\in L^1$ is derived from $B_\alpha = \sum_j \beta_{\alpha,j-1}\log\frac{1}{x_{\alpha,j}}\in L^1$ proven in \cite[Corollary 13]{MoCaMa_99}.
\end{proof}
\begin{rem}\label{prototype}
From the above lemma, one easily deduces the behaviour of $\tilde{W}_\alpha$ on a symmetric neighbourhood of the origin:
\begin{itemize}
    \item If $\alpha < 1$, then $\tilde{W}_\alpha$ is even on a neighbourhood of $0$, and the above expansion holds also for negative values:
    $$   \int_{0}^{x}\tilde{W}_\alpha (t)dt = -x\log|x| + x+ o(x)  \ \ \ \mbox{ as } x\to 0.   $$   
    \item However, this is not the case for $\alpha =1$; indeed, in this case, 
    $$\int_0^x \tilde{W}_1(t) dt = -|x|\log|x| + O(x)  \mbox{ as } x\to 0.$$
\end{itemize}
\end{rem}
In order to prove the second claim of the above remark, let us observe that if $x<0$, then one can use the functional equation $\tilde{W}_1(x)= \tilde{W}_1(x+1)= -\log(1+x) -(1+x)\tilde{W}_1(|x|/(1+x))$  and change of variable to get
\begin{eqnarray*}
\int_{0}^{x}\tilde{W}_1 (t)dt = -\int_{x}^{0} \tilde{W}_1 (t)dt &=& O(x^2) +\int_{x}^{0} (1+t) \tilde{W}_1 (-t/(1+t))dt\\
&=&O(x^2)+\int_0^{\frac{-x}{1+x}}\frac{1}{(1+y)^3}W_1(y)dy\\
&=&x\log |x| + O(x)    \mbox{ for } x \to 0^-.
\end{eqnarray*}  
Let us anticipate that the behaviour of $\tilde{W}_1$ near the origin leads to the failure of BMO property (as we shall soon see in  Lemma \ref{BMOfailure}).

In order to discuss the different asymptotic properties of $\tilde{W}_\alpha$ at rational points, we give a couple of definitions as follows.
\begin{dfn}
Let $w\in L^1(a,b)$ and $\xi \in (a,b)$; we say that $\xi$ is a singularity of {\em type A} if one of the following conditions holds
\begin{enumerate}
%\noindent $\rm{(i)}$

\item[\rm{($A_+$)}] $\lim_{h \to 0} \frac{1}{|h|}\int_\xi^{\xi +h} w(t) dt = +\infty$,
\item[\rm{($A_-$)}] $\lim_{h \to 0} \frac{1}{|h|}\int_\xi^{\xi +h} w(t) dt = -\infty$.
\end{enumerate}
We say that $\xi$ is a singularity of {\em type B} if one of the following conditions holds
\begin{enumerate}
\item[\rm{($B_+$)}] $\lim_{h \to 0} \frac{1}{h}\int_\xi^{\xi +h} w(t) dt = +\infty$,
\item[\rm{($B_-$)}] $\lim_{h \to 0} \frac{1}{h}\int_\xi^{\xi +h} w(t) dt = -\infty$.
\end{enumerate}
\end{dfn}
By Remark \ref{prototype}, the point $\xi=0$ is a type A singularity for $\tilde{W}_1$, while for $\alpha\in (0,1)$, $\xi=0$ is the prototype of type B singularity.

Type A and B are mutually exclusive conditions at a point $\xi$, and the multiplication by the function $\sigma(x):={\rm sign}(\xi-x)$ produces a switch between type A and B. Even if in general, types A and B do not cover all possible singularities of an $L^1$ function, this definition will well describe the behaviour of $\tilde{W}_\alpha$ at rational values. 

\begin{thm}\label{singclass}
Let $\alpha\in [0,1]$ and $\xi \in [\alpha-1,\alpha)\cap \Q$.
\begin{enumerate}
    \item[\rm{(i)}] If $\{\alpha, \alpha-1\} \cap \{T^k_\alpha(\xi), k\in \N\}=\emptyset $, then $\xi$ is a type B singularity for $\tilde{W}_\alpha$.
    \item[\rm{(ii)}]Otherwise, $\alpha \in \Q\cap [0,1]$; and in this latter case, $\xi$ is a type B singularity for $\tilde{W}_\alpha$ if and only if $\alpha$ belongs to a matching interval of even index.
\end{enumerate}
\end{thm}
When the condition of Theorem \ref{singclass}-(i) holds, we will say that $\xi$ is \emph{$\alpha$-regular}. Note that if $\alpha\in (0,1]\setminus \Q$, then every $\xi \in \Q$ is $\alpha$-regular. On the other hand, if some  $\xi \in \Q$ is not $\alpha$-regular, then $\alpha \in \Q$, and it belongs to some matching interval $J$; in this latter case, $\xi$ is a type A   (resp. B) singularity for $\tilde{W}_\alpha$ if the matching index of $J$ is odd (resp. even). In particular, if $\alpha\in [0,1]\cap \Q$ is a rational parameter belonging to a matching interval of odd index, then $\alpha$ is a type A singularity of $\tilde{W}_\alpha$.

The above consideration, together with the following general principle, will be the main tool to show that BMO condition fails for some parameters.
\begin{lem}\label{BMOfailure}
Let $w\in L^1(a,b),$ and let $\xi \in (a,b)$ be a type A singularity for $w$. Then,
\begin{enumerate}
    \item[\rm{(i)}] for every $\varepsilon>0$, there exist $x^+\in (\xi,\xi+\varepsilon)$ and $x^-\in (\xi-\varepsilon,\xi)$ such that $$\int_{x^-}^{x^+}w(t)dt=0,$$ and
    \item[\rm{(ii)}] $w \notin \mathrm{BMO}$.
\end{enumerate}
\end{lem}
We shall also need another lemma, which will be very useful in combination with the functional equation.
\begin{lem}\label{coc}
Let $w\in L^1(a,b),$ let $\xi \in (a,b)$, and let $\beta, \phi$ be two smooth functions  such that 
    \begin{itemize}
        \item  $\beta(\xi) \neq 0$,
        \item  $\phi'(\xi) \neq 0$ (hence $\phi$ is locally invertible near $\xi$).
    \end{itemize}
    If $\phi(\xi)\in (a,b)$ is a singularity of type B (resp. type A) for $w$, then the function $g(x):=\beta(x)w(\phi(x))$ has a singularity of type B (resp. type A) at $\xi$.
\end{lem}

% Lemma \ref{BMOfailure} will be the tool to prove the failure of BMO for $\tilde{W}_\alpha$ when $\alpha$ is a rational parameter in a 
% matching interval of odd index; in this case we shall apply it choosing $\xi =\alpha$ (rather than $\xi =0$).

Before proving our claims, let us show what happens for $\alpha \in (g,(5-\sqrt{13})/2)$. 
We have $1-\alpha\in(\frac{1}{4-\alpha},\frac{1}{2+\alpha})$ and $\frac{1}{\alpha}-1\in(\frac{1}{3-\alpha},\frac{1}{1+\alpha})$.
If $\epsilon$ is sufficiently small and $\alpha + \epsilon > x>\alpha$, then, by the functional equation \eqref{gen} for $K=0$, we get
$$ \tilde{W}_\alpha (x)= \tilde{W}_\alpha (1-x)=-\log(1-x)-(1-x)\tilde{W}_\alpha\left(3-\frac{1}{1-x}\right).$$
Let us set $\phi(x):=3-\frac{1}{1-x}=\frac{2-3x}{1-x}$, using the functional equation \eqref{gen} for $K=1$, we get, for $\alpha-\epsilon < x<\alpha$,
\begin{eqnarray*}
\tilde{W}_\alpha (x)&=&-\log x - x\log\left(\frac{x}{1-x}\right)+ (1-x) \tilde{W}_\alpha (\phi(x)).
\end{eqnarray*}
Therefore, in a neighbourhood $(\alpha-\epsilon, \alpha + \epsilon)$, we have
% $$\tilde{W}_\alpha(x)=\begin{cases}
% -\log(1-x)-(1-x)\tilde{W}_\alpha(\phi(x)) & x>\alpha\\
% -\log x - x\log(\frac{x}{1-x})+ (1-x) \tilde{W}_\alpha (\phi(x)) & x<\alpha 
% \end{cases}$$
$$ \tilde{W}_\alpha (x)= h(x) + {\rm sgn}(\alpha-x)(1-x)\tilde{W}_\alpha(\phi(x)) \ \ \ \ \mbox{ with }\ h(x)= \begin{cases}
 -\log(1-x), & x>\alpha,\\
 -\log x - x\log(\frac{x}{1-x}), & x<\alpha. 
 \end{cases}$$ 

We note that 
\begin{itemize}
    \item $h(x) = O(1)$ for $x\to \alpha$, hence it does not change the kind of singularity of $\tilde{W}_\alpha$ at $\alpha$; 
    \item by Lemma \ref{coc}, the singularity of the function $g(x):= (1-x)\tilde{W}_\alpha(\phi(x))$ at $\alpha$ is of the same kind as the singularity of 
    $\tilde{W}_\alpha$ at $\phi(\alpha)$;
    \item by Lemma \ref{coc}, the function $\tilde{W}_\alpha$ has a type (B) singularity at $\xi=\phi(\alpha)$. This is trivial if $\alpha=2/3$ (since $\phi(\alpha)=0$), and in the other cases, it can be easily seen using the functional equation \eqref{gen} with the smallest integer $K$ such that $T_\alpha^K(\xi)=0$ (indeed,  $T_\alpha^K(x)$ will be a smooth fractional transformation for $x$  in a neighbourhood of $\xi=\phi(\alpha)$);
    \item the function ${\rm sgn}(\alpha-x)g(x)$ has a type (A) singularity at $\alpha$.
\end{itemize}
Therefore we can conclude that $\tilde{W}_\alpha$ has a type (A) singularity at $\alpha$, and is not BMO by Lemma \ref{BMOfailure}.
The same argument applies to the other values in the interval $(g,1]$, the only thing that can change is the analytic form of the fractional transformation $\phi$.

\begin{proof}[Proof of Theorem \ref{optimal}]
The claim (i) is an immediate consequence of the above discussion, together with Lemma \ref{BMOfailure}. As  for the claim (ii), we already observed that each of the rational values
$u_m = [0;2,1^{2m-1}]$  is the pseudocenter of a matching interval of index $+1$, 
and $u_m \to 1-g$ as $m \to +\infty$, and by Theorem ~\ref{singclass}-(ii)?) and Lemma \ref{BMOfailure}, we get that $\tilde{W}_{u_m} \notin \mathrm{BMO}$ for all $m\in \N$, and this concludes the proof  of Theorem \ref{optimal}.
\end{proof}

\subsection{Technical proofs }
In this section, we present technical proofs of Lemma~\ref{BMOfailure}, Lemma~\ref{coc} and Theorem~\ref{singclass}.
\begin{proof}[Proof of Lemma \ref{BMOfailure}]
WLOG, we may assume that $w$ satisfies Condition ($A_+$). Moreover, we may also assume that $\varepsilon>0$ is such  that $(\xi-\varepsilon,\xi+\varepsilon)\subset (a,b)$ and that $\int_{\xi-\varepsilon}^{\xi} w(t)dt <0$ and $\int_{\xi}^{\xi+\varepsilon} w(t)dt>0$. 
  
  Let us consider $I(\varepsilon):=\int_{\xi-\varepsilon}^{\xi+\varepsilon} w(t)dt$, if $I(\epsilon)=0$, then assertion (i) holds; if not, let us assume $I(\varepsilon)>0$, and consider the function $g(x):=\int_{\xi-\epsilon}^{\xi+x} w(t) dt$: $g$ is  continuous on $[0,\epsilon]$, $g(0)<0$ and $g(\varepsilon)>0$, so by the intermediate value theorem, there is $x^+\in (0,\varepsilon)$ such that $g(x^+)=\int_{\xi-\varepsilon}^{x^+} w(t) dt=0$. An analogous argument works if $I(\varepsilon)<0$, and this concludes the proof of (i).

 To prove that $w\notin \mathrm{BMO}$, let $M$ be any constant; we can choose $\varepsilon>0$ so that
 $$\frac{1}{|x|}\int_\xi^{\xi +x} w(t) dt\geq M, \ \ \ \ \forall |x|<\varepsilon.$$
 Let us also fix $x^+\in (\xi,\xi+\varepsilon)$ and $x^-\in (\xi-\varepsilon,\xi)$ such that $\int_{x^-}^{x^+}w(t)dt=0$, and let us estimate the quantity
 $\frac{1}{|I|}\int_I |w(t)-w_I|dt$, where $I=[x^-,x^+]$ and $w_I=\frac{1}{|I|}\int_I w(t) dt =0$:
 $$\frac{1}{x^+-x^-} \int_{x^-}^{x^+} |w(t)|dt \geq \frac{1}{x^+-x^-}[M(\xi-x^-)+M(x^+-\xi)] = M.$$
 This ends the proof of (ii).
\end{proof}

\begin{proof}[Proof of Lemma \ref{coc}] Let us first consider the case $\phi(x)=x$; let us set $W(x):=\int_\xi^{\xi +x} w(t) dt$, and note that since $w\in L^1$, the function $W$ is continuous. Hence, integrating by parts, we get
$$  \int_\xi^{\xi +x} \beta(t) w(t) dt= W(\xi+x)\beta(\xi+x)-\int_\xi^{\xi +x} W(t)\beta'(t) dt.$$
Note that the second term in the righthand side is $O(x) $ as $x\to 0$. Thus, if $w$ has a type A singularity at $\xi$, we get that
$$\frac{1}{|x|}\int_\xi^{\xi +x} \beta(t)w(t) dt=\frac{1}{|x|}W(\xi+x)\beta(\xi+x)+ O(1) \ \ \ \mbox{ for } x\to 0\text{;}     $$
hence, the product $\beta w$ also has a type A singularity at $\xi$. The case of type B singularity goes through in the very same way.

To handle the general case, it is enough to use a change of coordinates; setting $s=\phi(t)$ and $\psi:=\phi^{-1}$, we have
$$  \int_\xi^{\xi +x} \beta(t) w(\phi(t)) dt=\int_{\phi(\xi)}^{\phi(\xi +x)} \beta(\psi(s)) w(s)  \frac{1}{\phi'(\psi(s))} ds,
$$
and we get our claim using the previous point, with $\Tilde{\beta}(s):=\frac{\beta(\psi(s))}{\phi'(\psi(s))} $ instead of $\beta$.
\end{proof}

\begin{proof}[Proof of Theorem \ref{singclass}] Let us first prove claim (i); we shall consider $0<\alpha<1$ (otherwise, the hypotheses are not met).
Since $\xi\in \Q$, there is a smallest $k_0 \in \N$ such that $T_\alpha^{k_0+1}(\xi)=0$; so, by the functional equation \eqref{gen}, we get
$$ \tilde{W}_\alpha(x)=\tilde{W}_\alpha^{(k_0)}(x) +(-1)^{k_0+1}\tilde\beta_{\alpha,k_0}(x)\tilde{W}_\alpha(T_\alpha^{k_0+1}(x))  \  \mbox{ with }\ \tilde{W}_\alpha^{(k_0)}(x)=\sum_{k=0}^{k_0}(-1)^k \tilde\beta_{\alpha,k-1}(x) \log \frac{1}{|T_\alpha^k(x)|}.$$

Since 
$ T_\alpha^k(\xi)\notin \{ \alpha,\alpha -1\}$ for $k\leq k_0$, we get that, in a neighbourhood of $\xi$,  $T_\alpha^{k_0+1}$ coincides with a fractional transformation $\phi$; for the same reason the term $\tilde{W}_\alpha^{(k_0)}$ is smooth in a neighbourhood of $\xi$ (hence it's irrelevant for the type of singularity); therefore by Lemma \ref{coc} the singularity in $\xi$ is of the same type as in $\phi(\xi)=0$, namely it is type B.

\medskip

For the proof of claim (ii), let us first point out that the hypotheses of (i) are met whenever $\alpha \notin \Q$ or when $\alpha\in \Q$ but $\mathrm{den}(\alpha)>\mathrm{den}(\xi)$, where ``den'' denotes the denominator of a rational number. Thus if we are in case (ii), we have that $\alpha \in \Q$ and den$(\alpha)\leq$den$(\xi)$.

If $\xi \notin \{\alpha-1, \alpha\}$, then there is $k_0\geq 1 $ such that $T_\alpha^{k_0}(\xi)=\alpha-1$ and  we can write
$$\tilde{W}_\alpha(x)=\tilde{W}_\alpha^{(k_0-1)}(x)+(-1)^{k_0}\beta_{k_0-1}(x)\tilde{W}_\alpha(T_\alpha ^{k_0}(x)).$$
However, since $T^k_\alpha(\xi)\notin \{\alpha, \alpha-1\}$ for $k<k_0$, $T_\alpha^{k_0}$ coincides with a fractional transformation $\phi$ on a neighbourhood of $\xi$, hence by Lemma \ref{coc}, $\xi$ and $\alpha-1$ are singularities of the same type for $T_\alpha$.

Therefore to prove claim (ii), it is enough to show that, for $\alpha\in \Q$,  $\alpha$ is of type A if and only if $\alpha$ belongs to a matching interval $J$ of odd index.
 To prove this, we use equation \eqref{localal} to express $W_\alpha$ near $\alpha$:
$$ W_\alpha(x)=h(x)+{\mathrm{sgn}(\alpha-x)}^{n-m}\beta(x) W_\alpha(\phi(x)),$$
where $h\in L^\infty$ on a neighbourhood of $\alpha$ and $\beta,\phi$ are smooth near $\alpha$.
 We first point out that  $\xi:=\phi(\alpha)$ is a type B singularity, since den($\xi)<$den$(\alpha)$, and hence falls in case (i) we just treated above.
  Then, by Lemma~\ref{coc}, the expression $\beta(x)\tilde{W}_\alpha(\phi(x))$ has a type B singularity in $\alpha$, hence we deduce
 that if $n-m$ is even, then the singularity of $\tilde{W}_\alpha$ at $\alpha$ is of type B as well, while if $n-m$ is odd, then it becomes of type A.
Since $h$ is bounded, this ends the proof of claim (ii).
\end{proof}

\section{BMO property for $\alpha\in [1-g,g]$}\label{quater} %% Seul Bee

The BMO property of $W_\alpha$ for $\alpha\in[\frac12,g]$ was proven in \cite[Theorem 2.3]{LeMar_24}.
In this section, we focus on $W_\alpha$ for $\alpha\in[1-g,\frac12]$.
Since BMO property is preserved when summing with an $L^\infty$ function, it suffices to prove the following theorem to establish Theorem~\ref{main}.
\begin{thm}\label{th:W_al.unif.bdd}
For $\alpha\in[1-g,\frac 12]$, $W_\alpha-W_{1/2}$ is uniformly bounded.
\end{thm}

The proof of the theorem follows the proof of \cite[Proposition 2.9]{LeMar_24}.
We provide a self-contained proof of the theorem to ensure readability.
To deduce that 
$B_\alpha(x)-\sum_{n=0}^\infty \frac{\log q_{\alpha,n+1}}{q_{\alpha,n}}$ is uniformly bounded in \cite[Theorem 8]{MoCaMa_99}, they proved 
$\sum_{n=0}^\infty\left|\beta_{\alpha,n-1}\log x_{\alpha,n}^{-1} -\frac{\log q_{\alpha,n+1}}{q_{\alpha,n}}\right|$
is uniformly bounded.
From this, we derive the following proposition which allows us to use the series $\sum_{n=0}^\infty(-1)^n\frac{\log q_{\alpha,n+1}}{q_{\alpha,n}}$ in place of $W_\alpha$.
\begin{pro}\label{pr:unif}
For $\alpha\in(0,1]$, $|W_\alpha(x) - \sum_{n=0}^\infty (-1)^n\frac{\log q_{\alpha,n+1}}{q_{\alpha,n}}|$ is uniformly bounded.
%Also, $|\sum_{n=0}^\infty\frac{\log q_{\alpha,n+1}}{q_{\alpha,n}}-\sum_{n=0}^\infty\frac{\log a_{\alpha,n+1}}{q_{\alpha,n}}|$ is uniformly bounded. 
\end{pro}

%For the sake of brevity, let us denote the $A_\alpha$-orbits and the $A_{1/2}$-orbits by 
%$x_k = A_\alpha^k(x_{\alpha,0})$ and $x_k' = A_{1/2}^k(x_{1/2,0}).$
%Accordingly, we denote the convergents, partial quotients associated with $A_\alpha$ by $\frac{p_k}{q_k}$ and $(a_k,\epsilon_k)$, and those associated with $A_{1/2}$ by $\frac{p_k'}{q_k'}$ and $(a_k',\epsilon_k')$.

In this section, for the sake of brevity, let us denote the $A_\alpha$-orbits and the $A_{1/2}$-orbits by 
$x_k = A_\alpha^k(x_{\alpha,0})$ and $x_k' = A_{1/2}^k(x_{1/2,0})$ for any $x\in\mathbb R$, 
Accordingly, we denote the convergents and the partial quotients associated with $A_\alpha$ by $\frac{p_k}{q_k}$ and $(a_k,\epsilon_k)$, and those associated with $A_{1/2}$ by $\frac{p_k'}{q_k'}$ and $(a_k',\epsilon_k')$.
We recall some fundamental properties of $\alpha$-continued fractions.
\begin{rem}\label{re:fund}
\begin{enumerate}
\item %Let us consider the regular continued fraction map $A_1(x)=\frac{1}{x}-\lfloor \frac{1}{x}\rfloor$.
For $x,x'\in(0,1]$ such that $A_1(x)=A_1(x')$, 
\begin{equation}\label{eq:A_al.coin}
\begin{cases}
A_\alpha(x)\le 1/2 & \text{ if and only if } A_\alpha(x)=A_{1/2}(x'),\\
A_\alpha(x)> 1/2 & \text{ if and only if } A_\alpha(x)+A_{1/2}(x')=1.
\end{cases}\end{equation}
%where $A_1(x)=\frac{1}{x}-\lfloor\frac{1}{x}\rfloor$ is the regular continued fraction map.
\item 
For $\alpha\in(0,\frac{1}{2})$, we have
\begin{equation}\label{eq:eps}
\begin{cases}(a_k,\epsilon_k)=(a_k',\epsilon_k')&\text{if }x_k=x_k'\text{ and }x_{k-1}=x_{k-1}',\\
(a_k,\epsilon_k)=(a_k'+1,\epsilon_k') &\text{if }x_k=x_k'\text{ and }\frac{1}{x_{k-1}}-1=\frac{1}{x_{k-1}'},\\
a_k=a_k'+1,~ \epsilon_k=-1,~\epsilon_k'=1&\text{if }x_k=1-x_k'\text{ and }x_{k-1}=x_{k-1}',\\
a_k=a_k'+2,~ \epsilon_k=-1,~\epsilon_k'=1& \text{if }x_k=1-x_k'\text{ and }\frac{1}{x_{k-1}}-1=\frac{1}{x_{k-1}'}.
\end{cases}
\end{equation}
\item\label{it:re:fund3}
For $\alpha\in(0,1]$, we have $q_n>\frac{1}{C\cdot \lambda^n}$, where $C>0$ and $0<\lambda<1$. 
Combining with $x^{-1/2}\log x\le 2/e$ for $x>0$, the series
$$\sum_{n=0}^\infty\frac{1}{q_n}\quad \text{ and }\quad\sum_{n=0}^\infty\frac{\log q_{n}}{q_{n}}$$
are uniformly bounded, see \cite[Eq. (3.9) and Proof of Theorem 8]{MoCaMa_99} and also \cite[Remark 1.7]{MaMoYo_97} for details.
\end{enumerate}
\end{rem}

\begin{pro}\label{pr:q_i}
For $x\in \mathbb R$, a tuple $S_i:=(x_i,x_i',q_i,q_i')$ is in one of the following four states:
\begin{enumerate}
\item[\rm{(A)}] $x_i=x_i'$ and $q_i=q_i'$,
\item[\rm{(B)}] $x_i=1-x_i'$, $x_i'\in(\alpha,\frac12]$ and $q_i-q_i'=q_{i-1}$,
\item[\rm{(C)}] $\frac{1}{x_i}-1=\frac{1}{1-x_i'}$ and $q_i-q_i'=-q_{i-1}'$,
\item[\rm{(D)}] $\frac{1}{x_i}-1=\frac{1}{x_i'}$ and $q_i=q_i'$.
\end{enumerate}
Moreover, the states change according to the diagram in Figure~\ref{fi:diag}.

\begin{figure} \centering
\begin{tikzpicture}
\node at (0,0) {$(\mathrm{A})$};
\node at (5,0) {$(\mathrm{B})$};
\node at (5,-2) {$(\mathrm{C})$};
\node at (0,-2) {$(\mathrm{D})$};
\draw[->] (0.5,0)--(4.5,0);
\draw[->] (4.5,-2)--(0.5,-2);
\draw[->] (5,-0.5)--(5,-1.5);
\draw[->] (0,-1.5)--(0,-0.5);
\draw[<->] (0.5,-1.5)--(4.5,-0.5);
\draw[->] (0,0.4) arc (0:270:0.5);
\draw[->] (5.5,-1.9) arc (90:-180:0.5);
\end{tikzpicture}
\caption{Diagram of the four states in Proposition~\ref{pr:q_i}.}
\label{fi:diag}
\end{figure}
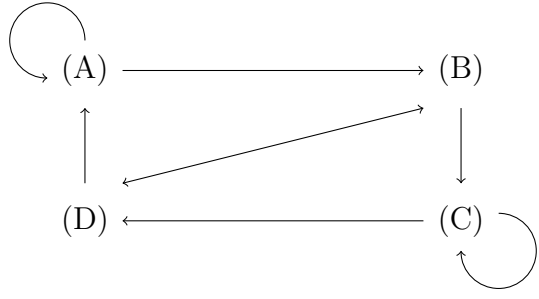
\end{pro}
\begin{rem}\label{prioriBC}
    If $S_i$ is in state $\mathrm{(B)}$, then $x_i \geq 1/2$. If it is in state $\mathrm{(C),}$ then $x_i \geq 1/3$.
\end{rem}

%\begin{rem}
%If $x_k=x_k'=0$, then $S_k\in\mathrm{(A)}$ and $S_{k+1}$ is not defined.
%\end{rem}
\begin{proof}[Proof of Proposition~\ref{pr:q_i}]
We will show it inductively. Assume that the statement holds for $0\le i\le k$.
Note that $S_0$ can only be in state (A) or (B). 
Moreover, (C) and (D) occur only in the chain of (B)-(C)-(C)-$\cdots$-(C)-(D).
Thus, it is unnecessary to consider (C) and (D) independently.

($\mathrm{(A)}\rightarrow \mathrm{(A)}\text{ or }\mathrm{(B)}$):
Suppose that $S_k\in\mathrm{(A)}$.
Then $S_{k-1}\in\mathrm{(A)}$ or $S_{k-1}\in\mathrm{(D)}$. Thus $x_k=x_k'$, $q_k=q_k'$, $q_{k-1}=q_{k-1}'$ and,  
by \eqref{eq:eps}, $\epsilon_k=\epsilon_k'$.
\begin{enumerate}
\item If $x_{k+1}\le 1/2$, then $x_{k+1}=x_{k+1}'$ by \eqref{eq:A_al.coin}. From \eqref{eq:eps}, $a_{k+1}=a_{k+1}'$, which implies $q_{k+1}=q_{k+1}'$. Thus $S_{k+1}\in\mathrm{(A)}$.
\item
For the case $x_{k+1}>1/2$, with a similar argument above, $x_{k+1}=1-x_{k+1}'$, $x_{k+1}\in(\alpha,\frac{1}{2}]$ and $a_{k+1}=a_{k+1}'+1$, which implies
$q_{k+1}=a_{k+1}q_{k}+\epsilon_kq_{k-1}=q_{k+1}'+q_k$.
Thus $S_{k+1}\in\mathrm{(B)}$.
\end{enumerate}

($\mathrm{(B)}\rightarrow\mathrm{(C)}\rightarrow \cdots \mathrm{(C)}\rightarrow\mathrm{(D)}\rightarrow\mathrm{(A)}\text{ or }\mathrm{(B)}$): Suppose that $S_k\in\mathrm{(B)}$.
Then $S_{k-1}\in\mathrm{(A)}$ or $S_{k-1}\in\mathrm{(D)}$.
Thus we have $x_k=1-x_k'$, $x_k'\in(\alpha,\frac{1}{2}]$
and
\begin{equation}\label{eq:q_kq_k-1}
q_k-q_k'=q_{k-1}, ~q_{k-1}=q_{k-1}', ~\epsilon_k=-1,~\epsilon_k'=1.
\end{equation}

Let $u_m=[0;2,1^{2m-1}]$ and $t_m=[0;2,1^{2m}]$, where $u_0=0$.
Since $u_m\uparrow 1-g$ and $t_m\downarrow 1-g$, we have
$$\left(0,1-g\right) = \sqcup_{m=1}^\infty (u_{m-1},u_m]\text{ and }\left(1-g,\frac12\right] = \sqcup_{m=1}^\infty(t_m,t_{m-1}].$$
Since $x_k'\in(\alpha,\frac12]$, there exists $m$ such that $x_k'\in(t_m,t_{m-1}]$.
If $m=1$, i.e. $x_k'\in\left(\max\{\frac{2}{5},\alpha\},\frac{1}{2}\right]$ and $x_k\in \left[\frac12,\min\{\frac35,1-\alpha\}\right)$, then $x_{k+1}'=\frac{1}{x_{k+}'}-2$ with $(a_{k+1}',\epsilon_{k+1}')=(2,1)$, and $x_{k+1}=2-\frac{1}{x_k}$ with $(a_{k+1},\epsilon_{k+1})=(2,-1)$.
Thus $\frac{1}{x_{k+1}}-1=\frac{-x_k+1}{2x_k-1} = \frac{x_{k}'}{1-2x_k'}=\frac{1}{x_{k+1}'}$.
By \eqref{eq:q_kq_k-1},
$q_{k+1} = 2q_k-q_{k-1}= 2q_k'+q_{k-1}'=q_{k+1}'$.
It means that $S_{k+1}\in\mathrm{(D)}$.

For $m>1$, we will show that 
\begin{equation}\label{eq:S_k+j}
S_{k+j}\in\mathrm{(C)}\text{ for }j=1,\cdots,m-1,\text{ and }S_{k+m}\in\mathrm{(D)}.
\end{equation}

For brevity, we denote by $f(x):= 3-\frac{1}{x}$.
Then we have
$f(t_i)=t_{i-1},~f(u_i)=u_{i-1}\text{ and }u_i = 2-\frac{1}{1-t_i}$ for all $i$.
Thus we have
\begin{equation}\label{eq:x_i}
\begin{cases}
x_{k+j}'=f(x_{k+j-1}') & \text{ and }(a_{k+j}',\epsilon_{k+j}')=(3,-1) \quad\text{ for }j=1,\cdots,m-1, \\
x_{k+m}'=\frac{1}{x_{k+m-1}'}-2& \text{ and }(a_{k+m}',\epsilon_{k+m}')=(2,1),\\
x_{k+1}=2-\frac{1}{x_k}, & \text{ and } (a_{k+1},\epsilon_{k+1})=(2,-1),\\
x_{k+j}=f(x_{k+j-1}) & \text{ and } (a_{k+j},\epsilon_{k+j})=(3,-1)\quad\text{ for }j=2,\cdots,m.
\end{cases}
\end{equation}

Let $g(x) = \frac{2x-1}{x-1}$. Note that $\frac{1}{x_i}-1=\frac{1}{1-x_i'}$ is equivalent to $x_i'=g(x_i)$.
We have 
$$x_{k+1}' = f(x_k') = f\left(1-\frac{1}{2-x_{k+1}}\right)= \frac{2x_{k+1}-1}{x_{k+1}-1}=g(x_{k+1}).$$ 
Since $f\circ g(x) = \frac{5x-2}{2x-1}=g\circ f(x)$,
the relation $x_{k+j-1}'=g(x_{k+j-1})$ implies that
$$x_{k+j}'=f(x_{k+j-1}')=f\circ g(x_{k+j-1}) = g\circ f(x_{k+j-1})=g(x_{k+j})\quad\text{ for }j=2,\cdots,m-1.$$
%Thus (C) holds for $(x_{i+j},x_{i+j}')$, $j=2,\cdots, m-1$.
%By combining $x_{k+m}'=\frac{1}{x_{k+m-1}'}-2$, $x_{k+m-1}'=g(x_{k+m-1})$ and $x_{k+m}=3-\frac{1}{x_{k+m-1}}$, 
From above, for $j=m$, we have
$$x_{k+m}'=\frac{1}{x_{k+m-1}'}-2=\frac{1}{g(x_{k+m-1})}-2=
%\frac{1}{g(\frac{1}{3-x_{k+m}})}-2
1-f(g(x_{k+m-1}))
=\frac{x_{k+m}}{1-x_{k+m}}.$$
%which implies $\frac{1}{x_{k+m}}-1=\frac{1}{x_{k+m}'}$.

On the other hand, combining \eqref{eq:x_i} and \eqref{eq:q_kq_k-1} with a recurrence relation of $q_i$, we have
\begin{equation}\label{eq:q'-q}\begin{cases}
q_{k+1}'-q_{k+1} & =3q_k'-2q_k+2q_{k-1}=q_k',\\
q_{k+j}'-q_{k+j} & = 3(q_{k+j-1}'-q_{k+j-1})-(q_{k+j-2}'-q_{k+j-2}) =q_{k+j-1}' \text{ for }2\le j\le m-1,\\
q_{k+m}'-q_{k+m} & =3(q_{k+m-1}'-q_{k+m-1})-(q_{k+m-2}'-q_{k+m-2})-q_{k+m-1}'=0,
\end{cases}\end{equation}
inductively.
Thus \eqref{eq:S_k+j} holds.

%The $(k+m)$th state is (D), i.e. 
%Let us consider the regular continued fraction map $A_1(x)=\frac{1}{x}-\lfloor \frac{1}{x}\rfloor$.

We will now show that $S_{k+m+1}\in\mathrm{(A)}$ or $\mathrm{(B)}$.
First, from $a_{k+m}'=2,~\epsilon'_{k+m-1}=-1$ as in \eqref{eq:x_i}, $q_{k+m}'-q_{k+m-1}'=q_{k+m-1}'-q_{k+m-2}'$. By \eqref{eq:q'-q}, we have 
\begin{equation}\label{eq:q_k+m-1}
q_{k+m}'-q_{k+m-1}'-q_{k+m-1}=0.
\end{equation}
 By $\frac{1}{x_{k+m}}=\frac{1}{x_{k+m}'}-1$, we have $A_1(x_{k+m})=A_1(x_{k+m}')$.
\begin{enumerate}
\item
If $A_{\alpha}(x_{k+m}')\le 1/2$, then $A_{1/2}(x_{k+m}')=A_\alpha(x_{k+m})$ by \eqref{eq:A_al.coin}.
By \eqref{eq:eps}, $a_{k+m+1}=a_{k+m+1}'+1$.
Recall that $q_{k+m}=q_{k+m}'$ in \eqref{eq:q'-q} and $\epsilon_{k+m}=-1$ and $\epsilon_{k+m}'=1$ in \eqref{eq:x_i}.
Combining with \eqref{eq:q_k+m-1}, we have
$q_{k+m+1}-q_{k+m+1}' = q_{k+m} - q_{k+m-1}- q_{k+m-1}'=0.$
Thus $S_{k+m+1}\in\mathrm{(A)}$.
\item
If $A_\alpha(x_{k+m}')>1/2$, then 
$A_{1/2}(x_{k+m}')= 1-A_{\alpha}(x_{k+m})\in(\alpha,\frac{1}{2})$ and
$a_{k+m+1}=a_{k+m+1}'+2$, thus
$q_{k+m+1}-q_{k+m+1}'=2q_{k+m}-q_{k+m-1}-q_{k+m-1}'=q_{k+m}'$
with a similar argument above.
Therefore, $S_{k+m+1}\in\mathrm{(B)}$.
\end{enumerate}
\end{proof}

\begin{proof}[Proof of Theorem~\ref{th:W_al.unif.bdd}]
The difference $W_\alpha-W_{1/2}$ is $\mathbb Z$-periodic and symmetric on $(0,\alpha)$.
By Proposition~\ref{pr:unif},
it is enough to show that  $\sum_{n=0}^\infty \left|\frac{\log q_{n+1}}{q_n}-\frac{\log q_{n+1}'}{q_n'}\right|$ is uniformly bounded for $x\in[0,1-\alpha]$.
We have
$$\frac{\log q_{n+1}}{q_n}-\frac{\log q_{n+1}'}{q_n'} = \frac{1}{q_n}\log\frac{q_{n+1}}{q_{n+1}'}+\left(\frac{1}{q_n}-\frac{1}{q_{n}'}\right)
\log(q_{n+1}').$$
%We note that
%\begin{equation}\label{eq:q_n/q_n+1}\frac{q_{i}'}{q_{i+1}'}=\frac{1}{a_{i+1}'+\frac{\varepsilon_i'}{a_i'+\frac{\varepsilon_{i-1}'}{\ddots+\frac{\varepsilon_1'}{a_1'}}}}\le \frac{1}{2+\frac{-1}{3+\frac{-1}{3+\ddots}}}=g \text{ for all }i.
%\end{equation}

By using the recurrence relation of $q_i'$, Hurwitz proved that $\frac{q_i'}{q_{i+1}'}\le g$ for all $i$ in \cite[\S 3]{Hur_89},
see also \cite[Satz 5.18 (B) in \S 43]{Pe_54} and \cite[p. 421]{Ja_85}.
By Proposition~\ref{pr:q_i}, $|q_i-q_i'|\le q_{i-1}$ for all $i$.
%Combining with \eqref{eq:q_n/q_n+1}, 
Thus we have
$$1-g \le \frac{-q_n'+q_{n+1}'}{q_{n+1}'}\le \frac{q_{n+1}}{q_{n+1}'}\le \frac{q_n'+q_{n+1}'}{q_{n+1}'}\le 2,$$
which implies that 
\begin{equation}\label{eq:log}\left|\log \frac{q_{n+1}}{q_{n+1}'}\right|\le \max\left\{\log 2,\log \frac{1}{1-g}\right\}=\log (g+2).
\end{equation}
In the proof of Proposition~\ref{pr:q_i}, we saw that if $|q_n'-q_n|=q_{n-1}'$, then $a_{n+1}'=2$ or $3$, see \eqref{eq:x_i}.
Thus $q_{n+1}'\le 3q_n'+q_{n-1}'\le 4q_n'$.
Then we have
\begin{align}\label{eq:1/q}
\left|\frac{1}{q_n}-\frac{1}{q_n'}\right| 
 =\frac{|q_n'-q_n|}{q_nq_n'}\le \frac{q_{n-1}'}{(q_n'-q_{n-1}')q_n'}=\frac{1}{\left(\frac{q_n'}{q_{n-1}'}-1\right)q_n'}
 \le \frac{1}{\left(\frac{1}{g}-1\right)q_n'}\le\frac{4}{gq_{n+1}'}.
\end{align}
By \eqref{eq:log}, \eqref{eq:1/q} and Remark~\ref{re:fund}-\eqref{it:re:fund3},
$$\sum_{n=0}^\infty \left|\frac{\log q_{n+1}}{q_n}-\frac{\log q_{n+1}'}{q_n'}\right|\le \sum_{n=0}^\infty\frac{\log(g+2)}{q_n}+\sum_{n=0}^\infty\frac{4}{g}\cdot\frac{\log q_{n+1}'}{q_{n+1}'}$$
is uniformly bounded.
%Thus we have $|W_\alpha(x)-W_{1/2}(x)|<C$ for $x\in[0,1-\alpha]$ since
%$$\left|W_\alpha(x)-\sum_{n=0}^\infty (-1)^n\frac{\log q_{n+1}}{q_n}\right|<C.$$
%If $x\in[1-\alpha,1]$, then $W_\alpha(x) = W_\alpha(1-x)$, $W_{1/2}(x)=W_{1/2}(1-x)$ and $1-x\in [0,1-\alpha]$, thus $|W_{\alpha}(x)-W_{1/2}(x)|<C$ for all $x$.
\end{proof}

Let us remark that for $\alpha$ near $1/2$, we can provide a much more precise estimate for the difference $W_\alpha - W_{1/2}$ exploiting the matching phenomenon: we give here an explicit example of the difference $W_{2/5}-W_{1/2}$, which has the following graphs.
\begin{figure}[h!]
    \centering
    \includegraphics[width=0.8\linewidth]{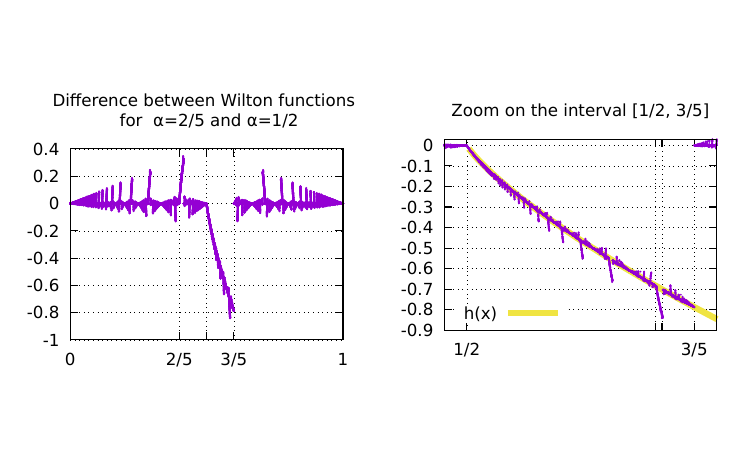}
    \caption{$W_\alpha - W_{1/2}$ for $\alpha=2/5$}
    \label{fig:difference}
\end{figure}

To interpret these pictures, let us first focus on some $x\in (1/2, 3/5)$ (see Figure~\ref{fig:difference}, right), so that we can use the classical regular continued fraction and write $x=[0;1,1,a+y]$ with $a\in \mathbb{N}$, $y\in (0,1)$ and $a+y>2$.
We first observe that, letting $x_0= x_{0,2/5} $ and $x'=x_{0,1/2}$, we get
$$
\begin{array}{ccc}
x_{0}=x=[0;1,1,a+y], & x_1=A_{2/5}(x)=1-A_1(x)=[0;a+1+y],     & A_1(x_1)=y;  \\
x'_{0}=1-x=[0;2,a+y], & x'_1=A_{1/2}(1-x)=A_1(1-x)=[0;a+y],     & A_1(x'_1)=y.
\end{array}
$$
Here we see that the orbits follow the diagram in Figure~\ref{fi:diag}.
%to evaluate $W_{2/5}-W_{1/2}$, 
More precisely, we start from state (B), and pass directly to state (D); after that we can either end up in state (A) or (B). 
Using the functional equation \eqref{gen}, we get
$$
\begin{array}{rcl}
W_{2/5}(x)     &=& -\log x - x\log \dfrac{x}{2x-1}+(2x-1)W_{2/5}(x_2),  \\
 W_{1/2}(x)=W_{1/2}(1-x)    &=& -\log (1-x) -(1-x)\log\dfrac{1-x}{2x-1} +(2x-1)W_{1/2}(x'_2).
\end{array}
$$
Therefore the difference can be written as
\begin{eqnarray}
   W_{2/5}(x)-  W_{1/2}(x)&=& h(x) + (2x-1)[W_{2/5}(x_2)-  W_{1/2}(x'_2)]  \ \ \ \mbox{ with } \\
      h(x) &=& -\log x - x\log \frac{x}{2x-1} 
   +\log (1-x) +(1-x)\log\frac{1-x}{2x-1},
\end{eqnarray}
 where either $x_2=x'_2$ (state (A)) or $x_2=1-x'_2$ (state (B)); in either case, one has that $\beta_1=\beta_1'.$

Note that $h$ is the function plotted with a yellow thick line in Figure~\ref{fig:difference}, which  follows the graph of the difference quite closely. Indeed, this can be explained easily: if we denote by $\tilde{B}=\{k \in \mathbb{N} \ \ : \ \ (x_k,x'_k) \  \mbox{is in state (B)}\}$, then we easily realize that
\begin{equation}\label{eq:diff}
    W_{2/5}(x)-  W_{1/2}(x)=\sum_{k\in \tilde{B}} \beta_{k-1}h(x_k).
\end{equation}

This formula holds in general, and explains the structure of the graph well; for instance, the part of the graph where $W_{2/5}(x)-  W_{1/2}(x)$ closely shadows $h(x)$ corresponds to a point for which $(x_k,x'_k)$ stays in state (A) for quite a few iterations, while intervals where the graph of the difference parts from that of $h$ corresponds to quick returns to state (B) (the largest ``hair'' shooting off the graph of $h$ for $x\in (7/12,18/31)$ corresponds to the case $x_2=1-x'_2=y$ with $y\in (1/2,3/5)$): namely a transition (D) to (B) without passing through state (A). One can use \eqref{eq:diff} together with the estimate of Proposition~\ref{prop} to prove rigorously that $\|W_{2/5}-  W_{1/2}\|_\infty < 1$ (which can be guessed from Figure~\ref{fig:difference}).

\section{Appendix}\label{apx}
In Section~\ref{tre}, we saw how the matching condition, whose relevance was first understood in connection with the study of the entropy of $\alpha$-continued fractions,  plays a key role in the mechanism leading to the failure of the BMO property. Actually, the techniques used in Section~\ref{quater} have the same flavour, even if the matching property is never explicitly mentioned. In fact, we can also use the intermediate results in Section~\ref{quater} to recover a very simple proof of the following non trivial fact:
\begin{pro}\label{pr:ent}
    The metric entropy of $A_\alpha$ is constant for $\alpha \in [1-g,g].$
\end{pro}
For $\alpha \in [1/2,g]$, this result is known since the eighties (\cite{Na81}), but extending it to $[1-g,1/2]$ is much harder. Indeed the proof of this result\footnote{Actually the result of \cite{CT13} is for the unfolded algorithm $T_\alpha$, but it is not difficult to see that the entropy of $A_\alpha$ and $T_\alpha$ is the same for all $\alpha\in [0,1]$.}  given by \cite{CT13} is quite sophisticated, the reason being that the range $[1-g,1/2]$ is split into countably many matching intervals.

However, the results in the previous section provide a straightforward proof of the ``hard" case $\alpha \in [1-g,1/2]$.
\begin{proof}[Proof of Proposition~\ref{pr:ent}]
It is well known that for $\alpha>0$ the map $A_\alpha$ has an (unique) ergodic absolute continuous invariant probability measure $\mu_\alpha$, and for a.e. $x\in [0,\bar{\alpha}]$ the invariant measure $\mu_\alpha$ and the metric entropy  can be computed as follows:
$$\mu_\alpha([a,b])=\lim_{n\to +\infty}\frac{\sum_{k=0}^{n-1}\chi_{[a,b]} (A^k_\alpha(x))}{n} , \ \ \ \  h_{\mu_\alpha}(A_\alpha)= 2 \lim_{k\to +\infty}\frac{1}{k} \log q_{\alpha, k}(x),$$
see \cite[Proposition 1]{NaNa_08}.
We shall call {\it typical} a value for which both above formulas hold.

Let $\alpha \in [1-g,1/2]$ be fixed, and let  us pick $x_0\in (0,1/3)$ which is typical both for $A_{1/2}$ and $A_\alpha$; resuming the notation of the previous section, we set
$x_k = A_\alpha^k(x_{0})$, $x_k' = A_{1/2}^k(x_{0})$ and define $\frac{p_k}{q_k}$, $\frac{p_k'}{q_k'}$ as the convergents of $x_0$ associated with $A_{1/2}$ and $A_\alpha$, respectively. 

Since 
$x_0$ is typical then there is an infinite set $J$ of indices $k$ such that $x_k\in (0,1/3)$ for  $k\in J$, and thus by Remark \ref{prioriBC} the pair $(x_k,x'_k)$ is either in state (A) or (D) for all $k\in J$ and $q_k=q'_k$ for all $k\in J$.
Therefore,
$$
h_{\mu_\alpha}(A_\alpha)= 2 \lim_{k\to +\infty}\frac{1}{k} \log q_k =
2 \lim_{\substack{k\to +\infty, \\k\in J}}\frac{1}{k} \log q_k 
=2 \lim_{\substack{k\to +\infty, \\ k\in J}}\frac{1}{k} \log q'_k=2 \lim_{k\to +\infty}\frac{1}{k} \log q'_k =h_{\mu_\alpha}(A_{1/2}).
$$

\end{proof}

%%%%
% Probably useless stuff (but you never know)
% \begin{thm}
% Let $\alpha\in [0,1]$, then
% \begin{enumerate}
%     \item[(i)] If $\alpha \notin \Q$ then every $\xi \in [\alpha-1,\alpha)\cap \Q$ then $\xi$ is a type B singularity for $W_\alpha$;
%     \item[(ii)] If $\alpha \in \Q\cap [0,1]$ and $\alpha$ belongs to a matching interval of even index, then every $\xi \in [\alpha-1,\alpha)\cap \Q$ then $\xi$ is a type B singularity for $W_\alpha$;
%     \item[(iii)] If $\alpha \in \Q\cap [0,1]$ and $\alpha$ belongs to a matching interval of odd index, then  $\xi \in [\alpha-1,\alpha)\cap \Q$  is a type B singularity for $W_\alpha$ if $\{\alpha, \alpha-1\} \cap \{T^k_\alpha(\xi), k\in \N\}=\emptyset $; otherwise $\xi$ is a singularity of type A, and in particular $\alpha$ itself is a type A singularity.
% \end{enumerate}
% \end{thm}

\Addresses

\end{document}